\documentclass[a4paper,12pt]{amsart}
\usepackage[paper = a4paper,left=4cm,top=1.5cm,right=1.5cm,bottom = 3cm,includehead]{geometry}
\usepackage{amssymb, newcent, enumerate}
\usepackage{amsmath}
\usepackage{mathrsfs}
\usepackage{graphics}
\usepackage{graphicx}
\usepackage{amsthm}
\usepackage[all]{xy}
\newtheorem{theorem}{Theorem}[section]
\newtheorem{proposition}[theorem]{Proposition}
\newtheorem{lemma}[theorem]{Lemma}

\newtheorem{corollary}[theorem]{Corollary}

\theoremstyle{remark}
\newtheorem{remarks}[theorem]{Remarks}

\newtheorem{definition}[theorem]{Definition}

\def\P{\mathbb{P}}
\def\RR{\mathbb{R}}

\def\V{\mathfrak{V}}
\def\a{\mathfrak{a}}
\def\b{\mathfrak{b}}
\def\c{\mathfrak{c}}

\def\MM{\tilde{M}}

\def\MMM{\tilde{\tilde{M}}}

\begin{document}
\title{A canonical Laplacian on the algebra of densities on a projectively connected manifold}
\author{Jacob George}
\address{School of Mathematics - University of Manchester, Oxford Road, Manchester M13 9PL}
\email{jgeorge@maths.manchester.ac.uk}

\begin{abstract}
On a manifold with a projective connection we canonically assign a
second order differential operator acting on the algebra of all
densities to any tensor density $S^{ij}$ of fixed weight $\lambda$.
In particular, this implies that on any projectively connected
manifold, a `bracket' (symmetric biderivation) on the algebra of
functions extends canonically to the algebra of densities.
\end{abstract}


\date{\today}
\maketitle
\section{Introduction}
In \cite{bialopap}, it was shown that a projective connection (or more accurately a `projective class') on a manifold $M$ can be canonically lifted to a linear connection on the total space of a principal $\RR$-bundle over $M$.  In addition, the algebra of all densities on $M$ can be considered as a subalgebra of the algebra of smooth functions on this bundle.  This paper builds on the results of \cite{bialopap} and further strengthens the close link between projective classes and the algebra of densities.

Projective connections, having been conceived in the 1920s by E. Cartan in \cite{Cartan} have flitted in and out of relative obscurity during the last century. They have emerged recently in a variety of contexts from integrable systems and projectively equivariant quantisation to so called Cartan algebroids. The concept itself has been recast and developed periodically most notably by Veblen and Thomas (\cite{1925thomas1,Thomas2,VeblenThomas}), Kobayashi and Nagano (\cite{KobayashiNagano,KobayashiTransf}), and Sharpe (\cite{SharpeBook}) among many others.  Of these refinements, we adopt the notion of \emph{projective equivalence class}, pioneered by Veblen and Thomas (and equivalent to the \emph{projective structures} of Kobayashi), as our definition of choice.  From here, to avoid confusion with other inequivalent definitions of projective connection, we exclusively use the term `projective class'.

Here, we briefly recall the notion of projective class and outline the relation between this and induced connections on the projectivisation of the tangent bundle.  A tour of Thomas' construction, the projective Laplacian and upper connections defined in \cite{bialopap} and an account of the results of Khudaverdian and Voronov follow and set the scene for the main results.  We prove that on a manifold $M$ endowed with a projective connection, a tensor density $S^{ij}$ of weight $\lambda$ can be extended to give an invariant homogeneous second order differential operator of weight $\lambda$ acting on the algebra of densities $\V(M)$ of all weights.  This immediately implies that on a manifold with a projective class, a symmetric biderivation (or \emph{bracket}) on $C^{\infty}(M)$ can be canonically extended to a bracket on $\V(M)$.  The operator and brackets constructed are then shown to be part of a family of operators and brackets respectively, indexed by the space of non-vanishing global densities and are compared with the results on projectively equivariant quantisation in \cite{Bouarroudj1}.  Finally, we discuss possible improvements and applications.
\section{Preliminaries}\label{prelim}
\subsection{Projective Connections}
The term `projective connection' refers to a number of disparate notions which are by no means equivalent.  A full survey being somewhat beyond the remit of this text, we refer the reader to \cite{crampin} and the references therein for a fuller picture.  Here we adopt the following definition.
\begin{definition}
A \emph{projective class} on a manifold $M$ is an equivalence class of torsion-free linear connections on $M$ having the same geodesics up to reparametrisation.
\end{definition}
Projective classes are also variously referred to as \emph{projective connections} and a \emph{projective structures}.  Despite the elegant geometric formulation, in practice, it is more useful to make use of the following classical proposition.
\begin{proposition}Let $M$ be a manifold endowed with two torsion-free linear connections $\nabla$ and $\bar{\nabla}$ with coefficients $\Gamma^{k}_{ij}$ and $\bar{\Gamma}^{k}_{ij}$ respectively.  Then $\nabla$ and $\bar{nabla}$ belong to the same projective equivalence class if and only if either
\begin{enumerate}[i)]
\item for some $1$-form $\omega$ and any two vector fields $X$ and $Y$,
\[\nabla_{X}Y - \bar{\nabla}_{X}Y := \omega(X)Y + \omega(Y)X,\]
\item or \begin{equation}\label{projcond}\Pi^{k}_{ij} := \Gamma^{k}_{ij} - \frac{1}{n+1}\left(\delta^{k}_{i}\Gamma^{s}_{sj} + \delta^{k}_{j}\Gamma^{s}_{is}\right) = \bar{\Gamma}^{k}_{ij} - \frac{1}{n+1}\left(\delta^{k}_{i}\bar{\Gamma}^{s}_{sj} + \delta^{k}_{j}\bar{\Gamma}^{s}_{is}\right) =: \bar{\Pi}^{k}_{ij}.\end{equation}
\end{enumerate}
\end{proposition}
We will also refer to $\Pi^{k}_{ij}$ as the \emph{coefficients of the projective class} or simply a \emph{projective class}.  The quantities $\Pi^{k}_{ij}$ were first introduced by T.Y.Thomas in \cite{1925thomas1} and were studied by the so called `Princeton School' in the 1920s.  A coordinate atlas can be chosen on $M$ for which $\Pi^{k}_{ij}$ disappear if and only if the corresponding transition functions belong to the projective group, that is the manifolds is locally projective.  The failure of $\Pi^{k}_{ij}$ to transform as a tensor is a possible definition of a multidimensional Schwarzian derivative.  For further details, see Ovsienko and Tabachnikov's book \cite{OvsienkoTabachnikov}.

One may conceive of a more natural seeming notion of projective connection.  Consider a vector bundle $E\to M $ equipped with a linear connection.  One could call the connection induced on the projectivisation of $E\to M $ `the associated projective connection'.  Let us explore this idea further.  A linear connection on a bundle $E\to M$ can be specified by a distribution in $TE$ deemed \emph{horizontal}.  Such a choice of distribution gives rise as usual to a lifting of tangent vectors from $TM$ to $TE$.  If $(x^{i}, \xi^{a})$ is a system of local coordinates on $E$ with $x^{1},\dots,x^{n}$ being base-like and $\xi^{0},\dots,\xi^{k}$ being fibre-like, the lift of $\frac{\partial}{\partial x^{i}}$ is given by
\[\widetilde{\left(\frac{\partial}{\partial x^{i}}\right)} = \frac{\partial}{\partial x^{i}} + \Gamma^{b}_{ia}\xi^{a}\frac{\partial}{\partial \xi^{b}}\]
where $\Gamma^{b}_{ia}$ are the connection coefficients.  These coefficients of course are specified by $\widetilde{\partial_{i}}(\xi^{a}) = \Gamma^{a}_{ib}\xi^{b}$.
Now let us projectivise the whole setup.  For a vector space $F$, let $\P(F)$ denote the associated projective space.  Then define $\P(E)\to M$ as the projectivisation of $E\to M$, that is, the bundle whose fibre at $x\in M$ is the projectivisation $\P(E_{x})$ of the fibre of $E$ at $x$, and whose transition functions are induced by those of $E$.  Choosing an horizontal distribution on $E$ induces a choice of distribution on $\P(E)$ via the derivative of the natural projection $\P(E)\to E$.  A lifting of vectors from $TM$ to $TE$ then induces a corresponding lifting from $TM$ to $T\P(E)$.  In inhomogeneous coordinates in which $\xi^{0} \neq 0$, the local coordinate system on $\P(E)$ is specified by $\frac{\xi^{a}}{\xi^{0}}$, $(1\le a\le k)$.  We have
\[\widetilde{\left(\frac{\partial}{\partial x^{i}}\right)}\left(\frac{\xi^{a}}{\xi^{0}}\right) = \Gamma^{a}_{ib}\frac{\xi^{b}}{\xi^{0}} - \frac{\xi^{a}}{\xi^{0}}\Gamma^{0}_{ib}\frac{\xi^{b}}{\xi^{0}}.\]
This expression vanishes for coefficients of the form $\Gamma^{a}_{ib} = \delta^{a}_{b}\gamma_{i}$.  For any choice of $\lambda\in \RR$ therefore, the lifts
\[\widetilde{\left(\frac{\partial}{\partial x^{i}}\right)} = \frac{\partial}{\partial x^{i}} + \left(\Gamma^{b}_{ia} + \lambda \delta^{b}_{a}\Gamma^{c}_{ic}\right)\xi^{a}\frac{\partial}{\partial \xi^{b}}\]
from $TM$ to $TE$ all induce the same horizontal distribution on $\P(E)$.  Indeed, such a choice of coefficients no longer specify a linear connection on $E$, rather a (projective) Ehresmann connection.  For a comprehensive account, refer to \cite{Hermann}.

The vanishing of the lifting for coefficients of the form $\Gamma^{a}_{ib} = \delta^{a}_{b}\gamma_{i}$ implies that the induced connection depends only upon the `trace free' part of $\Gamma^{a}_{ib}$, namely $\Phi^{a}_{ib}:=\Gamma^{a}_{ib} - \frac{1}{k+1}\Gamma^{c}_{ic}\delta^{a}_{b}$ where $\mathrm{rank}(E) = k+1$ (see \cite[Prop. 10, p38]{Manin}).  In the special case of $E = TM$, there is a choice between taking the first trace $\Gamma^{k}_{kj}$ or the second trace $\Gamma^{k}_{ik}$.  If we demand that both traces vanish, then the traceless part becomes
\[\Phi^{k}_{ij} = \Gamma^{k}_{ij} - \frac{1}{n+1}(\delta^{k}_{i}\Gamma^{s}_{sj} + \delta^{k}_{j}\Gamma^{s}_{is}) = \Pi^{k}_{ij}.\]
That two linear connections on $M$ induce the same linear connection on $\P(TM)$ is therefore implied by their being in the same projective class.
Before detailing the link between projective classes and densities, we give the following result from \cite{bialopap} the first part of which will form the foundation on which what follows is built.
\begin{theorem}\label{square}
Let $\Pi^{k}_{ij}$ be a projective class and $S^{ij}$ a tensor field on $M$.
\begin{enumerate}[i)]
\item Then
\[S^{ij}\partial_{i}\partial_{j} + \left(\frac{2}{n+3}\partial_{j}S^{ij} - \frac{n+1}{n+3}S^{jk}\Pi^{i}_{jk}\right)\partial_{i}\]
is a well defined differential operator on $M$, the \emph{projective Laplacian} on $M$.
\item There is a canonical associated canonical upper connection\footnote{Let $S^{ij}$ be a tensor on a manifold $M$ and $S^{\sharp}:T^{*}M \to TM$ the map $S^{\sharp}(\omega_{i}dx^{i}) = S^{ij}\omega_{j}\partial_{i}$.  Recall that an \emph{upper connection} or \emph{contravariant derivative} over $S^{ij}$ in a vector bundle $E\to M$ is a bilinear map $\nabla:\Omega^{1}(M)\times\Gamma(E)\to\Gamma(E)$ such that $\nabla^{f\omega}\sigma = f\nabla^{\omega}\sigma$ and $\nabla^{\omega}f\sigma = S^{\sharp}(\omega)(f)\sigma + f\nabla^{\omega}\sigma$ for any $1$-form $\omega$ and $f\in C^{\infty}(M)$.  Upper connections in the volume bundle arise as subprincipal symbols of second order differential operators, see \cite{oloii}.} over $S^{ij}$ in the bundle of volume forms whose coefficients are
    \[\Gamma^{i} = \frac{n+1}{n+3}\left(\partial_{j}S^{ij} + S^{jk}\Pi^{i}_{jk}\right).\]
\end{enumerate}
\end{theorem}
These formulae imply in particular that on manifolds which are locally isomorphic to real projective space, $S^{ij}\partial_{i}\partial_{j} + \frac{2}{n+3}\partial_{j}S^{ij}\partial_{i}$ is an invariant differential operator and $\frac{n+1}{n+3}\partial_{j}S^{ij}$ is a connection in the bundle of volume forms for any tensor field $S^{ij}$.
\subsection{The algebra of densities}
Let $M$ be a manifold.  A \emph{density of weight $\lambda$} on $M$ or \emph{$\lambda$-density} for short is a formal expression taking the form $\phi(x)|Dx|^{\lambda}$ in local coordinates.  Here $Dx$ denotes the local coordinate volume form and $\lambda\in\RR$.  The space of $\lambda$-densities forms a real vector space denoted $\V_{\lambda}(M)$.
Two densities may be multiplied using the rule
\[\phi(x)|Dx|^{\lambda}\cdot\psi(x)|Dx|^{\mu}:= \phi(x)\psi(x)|Dx|^{\lambda+\mu}.\]
Let us denote by $\V(M) = \bigoplus_{\lambda}\V_{\lambda}(M)$ the space of all densities.  Under the multiplication defined above this becomes the so called \emph{algebra of densities} introduced in \cite{oloii}.  Note that $\V_{0}(M) = C^{\infty}(M)$ is a subalgebra of $\V(M)$.

This algebra is endowed gratis with two natural features.  First, exploiting the close relation between densities and volume, there is a natural scalar product on $\V(M)$.  For $\phi\in\V_{\lambda}(M)$, $\psi\in\V_{\mu}(M)$, define
\begin{equation}\label{scalarproduct}
\langle\phi,\psi\rangle = \left\{\begin{array}{cc}\int_{M}\phi\psi& \mbox{if }\lambda+\mu = 1;\\0&\mbox{otherwise.}\end{array}\right.
\end{equation}
The second is a derivation, the \emph{weight operator}.  This is the linear operator $\V(M)\to\V(M)$ defined as having $\V_{\lambda}(M)$ as its $\lambda$-eigenspace for each $\lambda\in\RR$.  We turn our attention now to the Thomas construction which provides a link between densities and projective classes.

\subsection{The Thomas construction}
The object of Thomas' construction is to construct, a manifold $\tilde{M}$ for any manifold $M$ with the property that projective classes on $M$ can be canonically lifted to linear connections on $\tilde{M}$ (see \cite{Thomas2}, in which $\tilde{M}$ is denoted $^{*}\!M$).  The manifold $\tilde{M}$ is defined locally by attaching an additional coordinate to each chart on $M$ and specifying its coordinate changes.  If $x^{1},\dots,x^{n}$ are a local coordinate system on $M$, then let $x^{0}, x^{1},\dots, x^{n}$ be a coordinate system on $\tilde{M}$.  If under a coordinate change on $M$, $\bar{x}^{i} = f^{i}(x^{1},\dots,x^{n})$ then define $\bar{x}^{0} := x^{0} + \log J_{f}$ where $J_{f}$ is the absolute value of the determinant of the Jacobian matrix of $f$.  With this definition, a projective class $\Pi^{k}_{ij}$ gives rise to a linear connection $\tilde{\Gamma}^{\c}_{\a\b}$ on $\tilde{M}$ via the following formulae.  Here and throughout, Gothic indices range from $0$ to $n$ while Roman indices range from $1$ to $n$.
\[\tilde{\Gamma}^{k}_{ij} := \Pi^{k}_{ij},\quad \tilde{\Gamma}^{\a}_{0\b} = \tilde{\Gamma}^{\a}_{\b0} := \frac{-\delta^{\a}_{\b}}{n+1},\quad \tilde{\Gamma}^{0}_{ij} := \frac{n+1}{n-1}\left(\partial_{s}\Pi^{s}_{ij} - \Pi^{p}_{qi}\Pi^{q}_{pj}.\right)\]

Indeed, the manifold $\tilde{M}$ is a bundle over $M$ of rank $1$.  More specifically, it is the `oriented' frame bundle of the bundle of top forms on $M$.  It follows immediately that a density $\phi(x)|Dx|^{\lambda}$ on $M$ may be viewed as a function $\phi(x)e^{\lambda x^{0}}$ on $\tilde{M}$, the algebra $\V(M)$ as a subalgebra of $C^{\infty}(\tilde{M})$.  Through this correspondence, the weight operator $w$ can be interpreted as the tangent vector field $\frac{\partial}{\partial x^{0}}$ on $\tilde{M}$.

\subsection{Brackets and operators}
In \cite{oloii}, Khudaverdian and Voronov constructed from a bracket on $\V(M)$ a canonical generating operator subject to some natural conditions.  What follows is a brief account of this result and its relation to projective classes.

\begin{definition}\label{bracketdefn}
A \emph{bracket} on a unital commutative associative algebra  $A$ is a symmetric biderivation on $A$.  A bracket $\{\,\cdot\, , \,\cdot\,\}$ is said to be \emph{generated} by a differential operator $\Delta: A\to A$ if for all $a,b\in A$,
\[\{a,b\} = \Delta(ab) - a\Delta(b) - \Delta(a)b + ab\Delta(1).\]
A \emph{bracket} on a manifold $M$ is a bracket on $C^{\infty}(M)$.  A bracket on $\V(M)$ is said to have \emph{weight $\lambda$} if for $\phi\in\V_{\mu_{1}}(M)$, $\psi\in\V_{\mu_{2}}(M)$, $\{\phi,\psi\}\in\V_{\mu_{1} + \mu_{2} + \lambda}(M)$.
\end{definition}
For an operator $\Delta$ to generate a bracket, it must in fact be of second order in the algebraic sense.  Two second order operators generate the same bracket if and only if they differ by a first order operator on $A$.  Now the promised uniqueness result from \cite{oloii}.
\begin{theorem}\label{Koszul}
Let $A$ be endowed with an invariant symmetric scalar product $\langle\,\cdot\, ,\,\cdot\,\rangle$, i.e. $\langle a,bc\rangle = \langle ab,c\rangle$ $\forall\,a,b,c\in A$.  Then generating a given bracket, there is an unique operator $\Delta$ which is self adjoint with respect to the scalar product and such that $\Delta(1) = 0$.
\end{theorem}
Although the theorem can viewed in a purely algebraic context, it has wide reaching ramifications when applied to algebras of functions and densities, the latter being the main thrust of \cite{oloii}.  A bracket on $M$ is nothing but a symmetric tensor with two upper indices by another name: $S^{ij}\partial_{i}f\partial_{j}g = \{f,g\}$.  Adopting this correspondence, the first part of Theorem \ref{square} could be rephrased as, \emph{`On a manifold endowed with a projective class, every bracket has a canonical generating operator'}.

A bracket on $\V(M)$ of weight $\lambda$ is specified by three quantities:
\begin{equation}\label{bracketcomponents}
\begin{split}
S^{ij}|Dx|^{\lambda} &= \{x^{i},x^{j}\}\\
\gamma^{i}|Dx|^{\lambda +1}& = \{x^{i}, |Dx|\}\\
\theta|Dx|^{\lambda + 2}& = \{|Dx|, |Dx|\}
\end{split}
\end{equation}
Here we assume without loss of generality that our bracket is homogeneous and of weight $\lambda$.  Each of the objects $\gamma^{i}$ and $\theta$ have interesting geometric interpretations - see \cite{oloii}.
The theorem guarantees that generating each bracket on $\V(M)$, there is an unique operator which is self adjoint with respect to the scalar product (\ref{scalarproduct}) and constant free.  Calculating explicitly in coordinates, if the bracket is given by (\ref{bracketcomponents}), the generating operator is
\begin{equation}\label{THcanonical}
\Delta = |Dx|^{\lambda}\left(S^{ij}\partial_{i}\partial_{j} + 2\gamma^{i}w\partial_{i} + \theta w^{2} + (\partial_{j}S^{ij} + (\lambda - 1)\gamma^{i})\partial_{i} + (\partial_{a}\gamma^{a} + (\lambda -1)\theta)w\right).
\end{equation}

\section{The Main result}\label{classical}
This section will be devoted mainly to the explanation and derivation of the following result.
\begin{theorem}\label{main}Let $M$ be a manifold of dimension $n>1$ equipped with a projective class and a tensor density $S^{ij}$ of weight $\lambda\neq\frac{n+2}{n+1},\frac{n+3}{n+1}$.  There is a canonical second order differential operator of weight $\lambda$ acting on densities extending $S^{ij}$.
\end{theorem}
By \emph{tensor density of weight $\lambda$}, we refer to an object $S^{ij}$ such that $S^{ij}|Dx|^{\lambda}$ is a tensor.  The operator constructed above `extends' $S^{ij}$ in the sense that it is equal to $|Dx|^{\lambda}S^{ij}\partial_{i}\partial_{j} + \mbox{other terms}$. Theorem \ref{main} immediately implies the more palatable corollary:
\begin{corollary}\label{bracketextension}
Let $M$ be a manifold of dimension $n>1$ equipped with a projective class.  Then every bracket on $M$ can be extended canonically to a bracket on $\V(M)$.
\end{corollary}
\begin{proof}
Taking $\lambda = 0$, the tensor density $S^{ij}$ in the hypothesis of Theorem \ref{main} becomes a bracket on $M$.  The operator given by Theorem \ref{main} then generates a bracket on $\V(M)$ which extends the bracket $\{f,g\} = S^{ij}\partial_{i}f\partial_{j}g$.
\end{proof}
Let us take the first steps toward proving Theorem \ref{main}. In the previous section, an explicit formula for the canonical operator generating a given bracket on $\V(M)$ was given in terms of the bracket coefficients.  Regarding $\V(M)$ as a subalgebra of $C^{\infty}(\tilde{M})$, a bracket on $\V(M)$ is a bracket on $\tilde{M}$.  As mentioned, brackets on $M$ and symmetric tensors with two upper indices on $M$ are synonymous.  The same is true on $\tilde{M}$ and so a bracket on $\V(M)$ defines a tensor $\tilde{S}^{\a\b}$ on $\tilde{M}$.  Given a bracket defined by (\ref{bracketcomponents}), the tensor $\tilde{S}^{\a\b}$ decomposes as
\begin{equation}\label{d+1}\tilde{S}^{ij} = e^{\lambda x^{0}}S^{ij},\quad \tilde{S}^{0i} = \tilde{S}^{i0} = e^{\lambda x^{0}}\gamma^{i},\quad \tilde{S}^{00} = e^{\lambda x^{0}}\theta.\end{equation}
Suppose now that $\tilde{M}$ is endowed with a projective class $\tilde{\Pi}^{\a}_{\b\c}$ and a bracket given by (\ref{bracketcomponents}).  Then applying Theorem \ref{square} to $\tilde{M}$ yields a differential operator acting on densities:

\begin{equation}\label{densitysquare}\tilde{\Delta} = \tilde{S}^{\a\b}\partial_{\a}\partial_{b} + \left(\frac{2}{n+4}\partial_{\b}\tilde{S}^{\a\b} - \frac{n+2}{n+4}\tilde{S}^{\b\c}\tilde{\Pi}^{\a}_{\b\c}\right)\partial_{\a}.\end{equation}

Indeed, for the construction of $\tilde{\Delta}$, it is enough for $M$ to be equipped with a projective class since:
\begin{proposition}A projective class $\Pi^{k}_{ij}$ on $M$ induces a projective class $\tilde{\Pi}^{\a}_{\b\c}$ on $\tilde{M}$ defined as
\begin{equation}\label{inducedprojectiveclass}
\begin{split}
\tilde{\Pi}^{k}_{ij} := \Pi^{k}_{ij}, \quad \tilde{\Pi}^{k}_{i0} = \tilde{\Pi}^{k}_{0i} := \frac{-\delta^{k}_{i}}{(n+1)(n+2)},\quad \tilde{\Pi}^{0}_{i0} = \tilde{\Pi}^{0}_{0i} := 0,\\ \tilde{\Pi}^{0}_{ij} := \frac{n+1}{n-1}\left(\partial_{s}\Pi^{s}_{ij} - \Pi^{p}_{qi}\Pi^{q}_{pj}\right),\quad
\tilde{\Pi}^{k}_{00} := 0,\quad \tilde{\Pi}^{0}_{00} := \frac{n}{(n+1)(n+2)}.
\end{split}
\end{equation}
\end{proposition}
\begin{proof}Using the Thomas construction, the projective class $\Pi^{k}_{ij}$ on $M$ induces a linear connection on $\tilde{M}$.  The above expressions are given by taking the projective class of this connection.
\end{proof}
Substituting (\ref{d+1}) and (\ref{inducedprojectiveclass}) into (\ref{densitysquare}), the argument thus far can be summarized as:
\begin{lemma}
Let $M$ be a manifold endowed with a projective class and a bracket on $\V(M)$.  Then there is a canonical second order differential operator acting on densities generating the bracket.  If the bracket is given by (\ref{bracketcomponents}), the operator is given as
\begin{equation}\label{nasty}
\begin{split}
\tilde{\Delta} = e^{\lambda x^{0}}\left(S^{ij}\partial_{a}\partial_{b} + 2\gamma^{i}\partial_{j}\partial_{0} + \theta\partial_{0}^{2} + \left(\frac{2}{n+4}\partial_{j}S^{ij} + \frac{2(\lambda + n\lambda + 1)}{(n+1)(n+4)}\gamma^{i} -\frac{n+2}{n+4}S^{jk}\Pi^{i}_{jk}\right)\partial_{i}\right.\\\left. + \left(\frac{2}{n+4}\partial_{k}\gamma^{k} + \frac{2\lambda + 2\lambda n - n}{(n+1)(n+4)}\theta - \frac{(n+1)(n+2)}{(n-1)(n+4)}S^{ij}\left(\partial_{s}\Pi^{s}_{ij} - \Pi^{p}_{qi}\Pi^{q}_{pj}\right)\right)\partial_{0}\right).\end{split}\end{equation}
\end{lemma}
Note that in the above expression, $\partial_{0}$ is the weight operator defined above.  In the style of \cite{oloii}, $\tilde{\Delta}$ can be equivalently interpreted as a \emph{pencil} of operators $\tilde{\Delta}_{\mu}:\V_{\mu}(M)\to\V_{\mu+\lambda}(M)$ for $\mu\in\RR$ given by
\begin{equation*}\begin{split}
\tilde{\Delta}_{\mu} = e^{\lambda x^{0}}\left(S^{ij}\partial_{i}\partial_{j}  + \left(\frac{2}{n+4}\partial_{j}S^{ij} + \left(\frac{2(\lambda + n\lambda + 1)}{(n+1)(n+4)} + 2\mu\right)\gamma^{i} -\frac{n+2}{n+4}S^{jk}\Pi^{i}_{jk}\right)\partial_{i}\right.\\\left. + \frac{2\mu}{n+4}\partial_{k}\gamma^{k} + \mu\left(\frac{2\lambda + 2\lambda n - n}{(n+1)(n+4)} + \mu\right)\theta - \frac{\mu(n+1)(n+2)}{(n-1)(n+4)}S^{ij}\left(\partial_{s}\Pi^{s}_{ij} - \Pi^{p}_{qi}\Pi^{q}_{pj}\right)\right).\end{split}
\end{equation*}

Let us impose now the condition that $\tilde{\Delta}$ is self adjoint with respect to the scalar product (\ref{scalarproduct}).  Since it is already constant free, by Theorem \ref{Koszul} $\tilde{\Delta}$ must then be equal to the canonical operator (\ref{THcanonical}).  Equating coefficients,
\begin{equation}\label{gammathetaformulae}
\begin{split}
\gamma^{i} &= \frac{n+1}{n+3 - \lambda(n+1)}\left(\partial_{j}S^{ij} + S^{jk}\Pi^{i}_{jk}\right)\\
\theta &= \frac{n+1}{n+2 - \lambda(n+1)}\left(\partial_{s}\gamma^{s} + \frac{n+1}{n-1}S^{ij}\left(\partial_{s}\Pi^{s}_{ij} - \Pi^{p}_{qi}\Pi^{q}_{pj}\right)\right).
\end{split}
\end{equation}
Having obtained these expressions, they can be used to \emph{define} $\gamma^{i}$ and $\theta$.  Substituting (\ref{gammathetaformulae}) into (\ref{nasty}) gives Theorem \ref{main}.  Setting $\lambda = 0$ and substituting instead into (\ref{bracketcomponents}) proves Corollary \ref{bracketextension} directly without reference to Theorem \ref{main}.
\begin{remarks}
\begin{enumerate}[i)]
\item The operator is not well defined for the cases $\lambda = \frac{n+3}{n+1}$, $\frac{n+2}{n+1}$.  These cases are referred to by Ovsienko and Bouarroudj as \emph{resonant} - see \cite{Ovsienko,Bouarroudj1}.
\item As one would expect, if the bracket in the hypothesis of Theorem \ref{main} is a bracket on $M$, then we recover Theorem \ref{square}
\item In \cite{oloii} it was noted that in the case $\lambda = 0$, $\gamma^{i}$ defines an upper connection over $S^{ij}$ in the bundle of volume forms.  Indeed, setting $\lambda = 0$ in (\ref{gammathetaformulae}) recovers the expression given in Theorem \ref{square}.
\end{enumerate}
\end{remarks}
\subsection*{Further Remarks}
\subsubsection*{Iterated Thomas construction}

Consider the Jacobian matrix corresponding to a change of coordinates on $\MM$.  This is of the form
\[\left(\begin{array}{c|c}1&\frac{\partial\log J}{\partial x^{j}}\\\hline0&\frac{\partial \bar{x}^{i}}{\partial x^{j}}
\end{array}\right).\]
The modulus of the determinant of this matrix is therefore $\tilde{J} = J$.  It follows then that densities on $M$ are also densities on $\MM$ via the inclusion $\phi(x)|Dx|^{\lambda} \mapsto \phi(x)|\widetilde{Dx}|^{\lambda}\in\V(\MM)$ where $\widetilde{Dx}$ is the volume element on $\MM$.  All the above constructions can then be iterated, replacing $M$ with $\MM$ and $\MM$ with $\MMM$.  Continuing this process, Theorem \ref{main} and Corollary \ref{bracketextension} are the first in a family of results indexed by the natural numbers.

\subsubsection*{A special case}

In addition let us make an observation in the special case that $M$ is endowed with a density $\rho(x)|Dx|^{\sigma}$.  Then the canonical scalar product can be modified by defining
\begin{equation}\label{shiftedscalarproduct}
\langle\phi,\psi\rangle_{\rho,\sigma} = \left\{\begin{array}{cc}\int_{M}\phi\psi\rho& \mbox{if }\lambda+\mu + \sigma = 1;\\0&\mbox{otherwise.}\end{array}\right.
\end{equation}
Theorem \ref{Koszul} can be applied to $\V(M)$ with $\langle\phi,\psi\rangle_{\rho,\sigma}$ as the scalar product.  Turning the same handle, we obtain
\begin{theorem}Let $M$ be a manifold with a projective class $\Pi^{k}_{ij}$ a tensor density $S^{ij}$ of weight $\lambda$.  Then there is a family of second order differential operators $\Delta_{\rho,\sigma}$ extending $S^{ij}$ indexed by densities $\rho(x)|Dx|^{\sigma}$.
\end{theorem}
The formula is
\begin{equation}\label{nastier}\begin{split}
\Delta_{\rho,\sigma} = e^{\lambda x^{0}}\left(S^{i}_{j}\partial_{a}\partial_{b} + 2\gamma^{i}\partial_{j}\partial_{0} + \theta\partial_{0}^{2} + \left(\frac{2}{n+4}\partial_{j}S^{ij} + \frac{2(\lambda + n\lambda + 1)}{(n+1)(n+4)}\gamma^{i} -\frac{n+2}{n+4}S^{jk}\Pi^{i}_{jk}\right)\partial_{i}\right.\\\left. + \left(\frac{2}{n+4}\partial_{k}\gamma^{k} + \frac{2\lambda + 2\lambda n - n}{(n+1)(n+4)}\theta - \frac{(n+1)(n+2)}{(n-1)(n+4)}S^{ij}\left(\partial_{s}\Pi^{s}_{ij} - \Pi^{p}_{qi}\Pi^{q}_{pj}\right)\right)\partial_{0}\right).\end{split}
\end{equation}
where now
\begin{equation}\label{gammathetaformulaevolume}
\begin{split}
\gamma^{i} &= \frac{n+1}{n+3 - (\lambda + \frac{n+4}{n+2}\sigma)(n+1)}\left(\partial_{j}S^{ij} + S^{jk}\Pi^{i}_{jk} + \frac{n+4}{n+2}S^{ij}\partial_{j}\log\rho\right)\\
\theta &= \frac{n+1}{n+2 - (\lambda + \frac{n+4}{n+2}\sigma)(n+1)}\left(\partial_{s}\gamma^{s} + \frac{n+1}{n-1}S^{ij}\left(\partial_{s}\Pi^{s}_{ij} - \Pi^{p}_{qi}\Pi^{q}_{pj} \right)+ \frac{n+4}{n+2}\gamma^{s}\partial_{s}\log\rho\right).
\end{split}
\end{equation}

The existence of global densities $\rho(x)|Dx|^{\sigma}$ of non-zero weight is perhaps less restrictive than might be imagined.  Probing the bundle structure of $\tilde{M}\to M$ more precisely, it is clear that $\tilde{M}$ is a trivial $\RR_{+}$ bundle over $M$, where $\RR_{+}$ denotes the additive group of real numbers.  As such, there is some global section $M\to\tilde{M} = M\times \RR$ taking $x \mapsto (x,s(x))$.  Therefore we can consider in local coordinates on $M$ the densities $\rho = e^{\sigma s(x^{1},\dots,x^{n})}$ for arbitrary $\sigma\in\RR$.

Given a density as above, one can of course construct a $1$-density by raising to the appropriate power.  It was noted in \cite{oloii} that given such a density, $\rho$, the quantities $\gamma_{i} = -\partial_{i}\log\rho$ define a (flat) connection in the volume bundle.  The whole bracket on densities can be constructed in this case via $\gamma^{i} := S^{ij}\gamma_{i}$ and $\theta := \gamma^{i}\gamma_{i}$.  The bracket constructed is distinct from the bracket given above.

\subsubsection*{Projectively equivariant quantisation}

Theorem \ref{main} assigns to each tensor density $S^{ij}$ of weight $\lambda$ a second order differential operator acting on densities.  Such a procedure fits into the general programme dubbed \emph{projectively equivariant quantisation} by Lecomte and Ovsienko in their seminal paper \cite{Ovsienko}.  By a quantisation on $\RR^{n}$, we refer to a bijection from the space of tensor densities $S^{a_{1}\dots a_{k}}$ (or `symbols'), to the space of differential operators on $C^{\infty}(\RR)$ which preserves the principal symbol.  Requiring this bijection to be equivariant with respect to the action of $\mathrm{Diff(\RR^{n})}$ is too stringent and no such map exists.  However, restricting to the action of the projective group is more reasonable.  Indeed, in \cite{Ovsienko}, it is proved that not only does such a map exist, it is also unique.

The authors of \cite{Ovsienko} speculated on the possibility of a generalisation of this result in the case of manifolds endowed with a projective class.  Bordemann answered this question in \cite{Bordemann}, producing such a quantisation.  Bouarroudj produced explicit formulae for projectively equivariant quantisations in degree two and three (\cite{Bouarroudj1, Bouarroudj2}).  Later Mathonet and Radoux constructed an equivariant quantisation from a Cartan projective connection (\cite{MathonetRadoux1, MathonetRadoux2}) before Radoux showed (\cite{Radoux1}) that in this more general setup, there are many possible projectively equivariant quantisations.

It is most natural to compare the result here to the explicit formula in \cite{Bouarroudj1}.  In the case $\lambda = 0$, the operators agree, as pointed out to me by S. Bouarroudj.  However, in the general case, the operators disagree despite having the same resonant values for $\lambda$ (see \cite{Bouarroudj1} for more details).

\section{Further Discussion}
Given the results contained here, there are many avenues for further interest.  We give a few that immediately spring to mind.

\subsubsection*{Vector bundles}The projective connection coefficients $\Phi^{a}_{ib}$ defined in \S\ref{prelim} are defined for arbitrary vector bundles.  In the case of the tangent bundle, Theorem \ref{square} can be modified to give a projective Laplacian defined in terms of $\Phi^{k}_{ij}$ rather than $\Pi^{k}_{ij}$.  One avenue for further study is the generalisation of the results given here to arbitrary vector bundles.  This would involve an appropriate of modification of Thomas' construction, replacing $\tilde{M}$ with some other bundle.
\subsubsection*{Projectively equivariant quantisation} Although the notion of projectively equivariant quantisation has been alluded to, Theorem \ref{main} does not define a bona fide example, only giving the quantisation map in the case of second order symbols and operators.  Higher order analogues of Theorems \ref{square} and \ref{Koszul} apart from being interesting in their own right, would no doubt prove fruitful in this regard.
\subsubsection*{Poisson manifolds}  All brackets considered here are symmetric.  This immediately excludes the possibility of applying the constructions to Poisson manifolds and other antisymmetric structures.  On a supermanifold however, any odd symmetric bracket gives rise to an odd Poisson bracket canonically.  To study this situation, projective classes, the projective Laplacian and the Thomas construction would have to be replicated for supermanifolds.  This is not only possible, but is the subject for a text in preparation.
\subsection*{Acknowledgements:}  This work would not have been possible without countless discussions with my patient teachers Hovhannes Khudaverdian and Ted Voronov and their initial suggestion to investigate projective connections and differential operators.  A great deal of credit is also due to O. Little, R. Djabri, C. Walton and H. Zare who, albeit unwittingly, helped fashion the results presented here into a coherent form.
\bibliography{compare}
\bibliographystyle{plain}
\end{document}